\documentclass[11pt, twoside]{article}
\usepackage{amsfonts}
\usepackage{amsmath}
\usepackage{amssymb}
\usepackage{amsthm}
\usepackage[mathscr]{euscript}
\usepackage{mathrsfs}
\usepackage{indentfirst}
\usepackage{enumerate}
\usepackage{color}
\usepackage{lipsum}
\newtheorem{theo}{Theorem}[section]
\newtheorem{lem}{Lemma}[section]

\numberwithin{equation}{section}

\newcommand{\lbl}[1]{\label{#1}}
\allowdisplaybreaks %%让多行公式强制换页%%

\newcommand{\be}{\begin{equation}}
\newcommand{\ee}{\end{equation}}
\newcommand\bes{\begin{eqnarray}} \newcommand\ees{\end{eqnarray}}
\newcommand{\bess}{\begin{eqnarray*}}
\newcommand{\eess}{\end{eqnarray*}}
\newcommand{\bbb}{\begin{cases}}
\newcommand{\nnn}{\end{cases}}
\newcommand{\bea}{\begin{align*}}
\newcommand{\eea}{\end{align*}}

\newcommand\ep{\varepsilon}

\newcommand\kk{\left}
\newcommand\rr{\right}
\newcommand\dd{\displaystyle}

\newcommand\dx{{\rm d}x}

\newcommand\lm{\lambda}

\newcommand\yy{\infty}
\newcommand\qq{\eqref}
\newcommand\ff{, \ \ \forall \ }

\newcommand\ol{\overline}
\newcommand\ud{\underline}

\begin{document}\thispagestyle{empty}
\begin{center}
 {\Large Dynamics of a nonlinear infection viral propagation model}\\[1mm]
 {\Large with one fixed boundary and one free boundary\footnote{This work was supported by NSFC Grant 12171120}}\\[4mm]
 {\large Mingxin Wang\footnote{{\sl E-mail}: mxwang@hpu.edu.cn}}\\[0.5mm]
 {School of Mathematics and Information Science, Henan Polytechnic University, Jiaozuo, 454003, China}
\end{center}

\begin{quote}
\noindent{\bf Abstract.} In this paper we study a nonlinear infection viral propagation model with diffusion, in which, the left boundary is fixed and with homogeneous Dirichlet boundary conditions, while the right boundary is free. We find that the habitat always expands to the half line $[0, \yy)$, and that the virus and infected cells always die out when the {\it Basic Reproduction Number} $\mathcal{R}_0\le 1$, while the virus and infected cells have persistence properties when $\mathcal{R}_0>1$. To obtain the persistence properties of virus and infected cells when $\mathcal{R}_0>1$, the most work of this paper focuses on the existence and uniqueness of positive equilibrium solutions for subsystems and the existence of positive equilibrium solutions for the entire system.


 \end{quote}

 \pagestyle{myheadings}
\section{Introduction}\markboth{A nonlinear infection viral propagation model}{M.X. Wang}
{\setlength\arraycolsep{2pt}

In order to understand the pathogenesis of diseases and seek effective treatment measures, viral dynamics has always been research hotspots (\cite{Wei, NBHB,  Per}), which usually cannot be answered by biological
experimental methods alone but require the help of mathematical models. The basic model of virus dynamics is the following ordinary differential system (\cite{NBHB, Noba, BP, NoM, Wod})
 \bess\left\{\!\begin{array}{ll}
u'= \theta-au-buw,\\
v'=buw-cv,\\
w'=kv-qw,
 \end{array}\right.\eess
where $u$, $v$ and $w$ represent the population of uninfected cells, infected cells and viruses, respectively. In this model, susceptible cells are produced at rate $\theta$, die at rate $au$, and become infected at rate $buw$; infected cells are produced at rate of $buw$ and die at rate of $cv$; free viruses emerge from infected cells at a rate of $kv$ and are removed at a rate of $qw$.
This model is the simplest mathematical model of the interaction of uninfected cells, infected cells, and free viruses, since only the most basic relationships between these three species are incorporated. In nature, biological movement plays a vital role in many biological phenomena (\cite{Br03}). To investigate the impact of spatial dynamics on this model, the authors of \cite{WW07, Sta} extended this  model to include spatially random diffusion.

A nonlinear infection rate can happen due to saturation at high virus concentration, where the infectious fraction is so high that exposure is very likely. Moreover, with the increase of the virus concentration the living environment for cells becomes worse and worse. It is reasonable to think that the rate of infection for virus and the virion production rate for infected cells are both nonlinear. In some situations, the major spatial dispersal comes from the moving (diffusion) of viruses in vivo, while both the uninfected and infected cells are immobile (do not diffuse). The distribution of viruses and infected cells is a local range, which is small relative to the distribution of uninfected
cells. Besides, since the infected cells are caused by viruses, their distribution range is the same. On the other hand, in initial time, viruses are distributed over a local range $\Omega_0$ (initial habitat). Then they spread from boundary to expand their habitat as a result of the spatial dispersal freely.
Based on these observations and the {\it deduction of free boundary conditions} given in \cite{hdu12}, in the one dimensional case, Li et al. \cite{LLW} investigated the free boundary problem:
 \bess\begin{cases}
u_t=\theta-au-\dd\frac{buw}{1+w}=:f_1(u,w),  &t>0, \; x\in\mathbb{R},\\[2mm]
v_t=\dd\frac{buw}{1+w}-cv=:f_2(u,v, w), &t>0, \ g(t)<x<h(t),\\[2mm]
w_t-dw_{xx}=\dd\frac{kv}{1+w}-qw=:f_3(v,w), &t>0, \ g(t)<x<h(t),\\[1mm]
v(t,x)=w(t,x)=0, &t>0, \ x\notin(g(t),h(t)),\\
g'(t)=-\mu w_x(t,g(t)),\,\,\,h'(t)=-\beta w_x(t,h(t)), &t\ge0,\\
u|_{t=0}=u_0(x), \;x\in\mathbb{R};\;\;
(v, w)|_{t=0}=(v_0(x),w_0(x)), &-h_0\le x\le h_0,\\
 h(0)=-g(0)=h_0.
 \end{cases}\eess
They found that the virus cannot spread successfully when $\mathcal{R}_0=kb\theta/(acq)\le 1$, whether the virus successfully spread  depends on the initial value and parameters when $\mathcal{R}_0>1$.

In the situation of uninfected cells, infected cells and viruses all have spatial diffusion capabilities, and they have the same distribution range.
When one end of the habitat (e.g. the left end) is fixed and has homogeneous Dirichlet boundary conditions and the other end is a free boundary, the corresponding model becomes (write $(u,v,w)$ as $(u_1, u_2, u_3)$)
 \bes\begin{cases}
\partial_t u_1-d_1\partial_{xx}u_1=f_1(u_1,u_3),  &t>0, \; 0<x<h(t),\\
\partial_t u_2-d_2\partial_{xx}u_2=f_2(u_1,u_2,u_3), &t>0, \; 0<x<h(t),\\
\partial_t u_3-d_3\partial_{xx}u_3=f_3(u_2,u_3), &t>0, \; 0<x<h(t),\\
u_i=0,\;\;i=1,\,2,\,3, &t>0, \; x\notin(0,h(t)),\\
h'(t)=\dd\sum_{i=1}^3\mu_i\partial_x u_i(t, h(t)),\, &t>0,\\
 h(0)=h_0;\;\;u_i(0,x)=u_{i0}(x),\;\;i=1,\,2,\,3, &0\le x\le h_0,
  \label{1.3}\end{cases}\ees
where parameters are all positive constants, and initial function $u_{i0}$ satisfy
 \bess
u_{i0}\in C^2([0,h_0]),\;\;u_{i0}>0\;\;\text{in}\,(0,h_0),
\;\;u_{i0}(0)=u_{i0}(h_0)=0,\;\;i=1,2,3.
 \eess

The main purpose of this paper is to study the dynamics of \qq{1.3}.
By combining the methods in references \cite{WZjde18, Wdcds19}, we can prove the following theorem concerning the existence, uniqueness, regularity and estimates of global solution of \qq{1.3}.

\begin{theo}\lbl{th1.1} \, The problem \qq{1.3} admits a unique global solution $(u_1,u_2,u_3, h)$ and
 \bess
  u_i\in C^{1+\frac{\alpha}2,\,2+\alpha}((0,\yy)\times(0, h(t)),\;i=1,2,3; \;h \in C^{1+\frac{1+\alpha}2}([0,\infty)),\;\;h'(t)>0\;\;{\rm in}\;\;(0, \yy).
 \eess
Moreover, there exists $C_0>0$ such that $0<u_i\le C_0$ in $ [0,\yy)\times(0, h(t))$, and
  \bess
  \|u_i(t,\cdot)\|_{C^1([0,\,h(t)])}\leq C_0, \;\;\forall \; t\ge  1,\; i=1,2,3; \;\;\;\|h'\|_{C^{{\alpha}/2}([1,\yy))}\leq C_0. \eess
 \end{theo}

Take advantage of Theorem \ref{th1.1} and \cite[Lemma 8.7]{Wpara} it can be deduced that
 \[h_\yy=\lim_{t\to\yy}h(t)=\yy.\]

The paper is organized as follows. Section 2 concerns with positive solutions of the corresponding equilibrium system \qq{2.1} and its subsystems in half line $[0,\yy)$. We first prove the existence and uniqueness of bounded positive solutions for a general system \qq{2.4} in $[0,\yy)$ with two components and then construct the ordered coupled positive upper and lower solutions of \qq{2.1}. The ordered coupled positive upper and lower solutions constructed here can be used to prove the existence of the positive solution of \qq{2.1}, as well as to study the long time properties of the solution component $(u_1, u_2, u_3)$ of \qq{1.3}. Next, we will use the upper and lower solution methods and the topological degree theory in cones, respectively, to prove the existence of positive solutions of \qq{2.1}. Section 3 concerns with the long time properties of the solution components $(u_1, u_2, u_3)$ of \qq{1.3}.

There have been a lot works on the local diffusion equations and systems where only one end is a free boundary and the other one is fixed with homogeneous Dirichlet or mixed boundary condition, please refer to \cite{ZW18} for reaction-diffusion-advection equation, \cite{Wjde15, Wjfa16} for the logistic equation with sign-changing coefficient, \cite{Wjde14, Wcnsns15} for prey-predator model and \cite{WZjdde14} for competition model. For the work on non-local diffusion equations and systems where only one end is a free boundary, see the literatures \cite{LWZjde22, LLW-Z22, LLW-P24}. For the relevant results and progress on free boundary problems with local and/or non-local diffusions, the interested readers can refer to the survey paper \cite{Du22}.

\section{The equilibrium problem}

In this section we shall study the positive solutions of
 boundary value problems in half line:
  \bes\begin{cases}
-d_1U_1''=f_1(U_1,U_3), \; &0<x<\yy,\\
-d_2U_2''=f_2(U_1,U_2,U_3), &0<x<\yy,\\
-d_3U_3''=f_3(U_2,U_3), &0<x<\yy,\\
U_1=U_2=U_3=0, & x=0.
  \label{2.1}\end{cases}\ees
It is well known that the problem
  \bes\begin{cases}
 -d_2U_1''=\theta -aU_1,\;\;0<x<\yy,\\
 U_1(0)=0
 \end{cases}\lbl{2.2}\ees
has a unique bounded positive solution $\ol U_1(x)$, and $\ol U_1'(x)>0$ in $[0,\yy)$ and
$\lim_{x\to\yy}\ol U_1(x)=\theta/a$.

\begin{theo}\lbl{th2.1} If $\mathcal{R}_0={kb\theta }/(acq)\le 1$, then \qq{2.1} has no positive solution.
\end{theo}

\begin{proof} Let $(U_1, U_2, U_3)$ be a nonnegative solution of \qq{2.1}. Then $U_1(x)\le \ol U_1(x)$ by the comparison principle. So, $U_1(x)\le\theta/a$, and then  $(U_2, U_3)$ satisfies
 \bess\begin{cases}
-d_2U_2''\le f_2(\theta/a, U_2, U_3), &0<x<\yy,\\
-d_3U_3''=f_3(U_2, U_3), &0<x<\yy,\\
U_2=U_3=0, & x=0.
  \end{cases}\eess
Since ${kb\theta }/(acq)\le 1$, it is easy to show that the nonnegative solution of
\bess\begin{cases}
-d_2V_2''=f_2(\theta/a, V_2, V_3), &0<x<\yy,\\
-d_3V_3''=f_3(V_2, V_3), &0<x<\yy,\\
V_2'=V_3'=0, & x=0
  \end{cases}\eess
can only be zero. It follows that $(U_2, U_3)=(0,0)$ by the comparison principle.
\end{proof}

\subsection{The positive solutions of sub-systems of \qq{2.1}}\lbl{s2.1}

We first consider a general boundary value problem with two components. Assume that $\beta>0$ and satisfies
 $$bk\beta>cq.$$
Suppose that
 \bes
 \rho\in C^\alpha_{\rm loc}([0,\yy)),\;\;\rho\ge0 \;\;{\rm and}\;\; \lim_{x\to\yy}\rho(x)=\beta.
  \lbl{2.3}\ees
We first show that the problem
 \bes\begin{cases}
-d_2U_2''=f_2(\rho(x), U_2, U_3)=\dd\frac{b \rho(x)U_3}{1+U_3}-cU_2=:g_2(x,U_2, U_3), &0<x<\yy,\\[3mm]
-d_3U_3''=f_3(U_2, U_3)=\dd\frac{kU_2}{1+U_3}-qU_3=:g_3(U_2, U_3), &0<x<\yy,\\[1mm]
U_2=U_3=0, & x=0
  \end{cases}\label{2.4}\ees
has a unique bounded positive solution.

Take $l>0$ and consider the following boundary value problem
\bes\begin{cases}
-d_2U_2''=g_2(x,U_2, U_3), &0<x<l,\\
-d_3U_3''=g_3(U_2, U_3), &0<x<l,\\
U_2=U_3=0, & x=0, l.
  \end{cases}\label{2.5}\ees
Let $\lm_1(l)>0$ be the principle eigenvalue of the following eigenvalue problem
 \bes\begin{cases}
 -\psi''=\lm_1\psi,\;&l<x<2l,\\
 \psi=0,&x=l, 2l.
 \end{cases}\lbl{2.7}\ees
Then $\lim_{l\to\yy}\lm_1(l)=0$. By use of \qq{2.3} and $bk\beta>cq$, there exist $0<\ep\ll1$ and $l^*\gg 1$ such that
 \bes
  \rho(x)>\beta-\ep \;\;{\rm in}\;\; [l^*, \yy)
 \lbl{2.8}\ees
and
  \bes
 bk\kk(\beta-\ep\rr)>(c+d_2\lm_1(l))(q+d_3\lm_1(l)),\;\;\forall\, l\ge l^*.
 \lbl{2.6}\ees

\begin{theo}\lbl{th2.2} Let $bk\beta>cq$ and $l^*$ be given in the above. Then the problem \qq{2.5} has a unique positive solution $(U_{2l},  U_{3l})$ when $l> 2l^*$, and
 \bes
 \lim_{l\to\yy}(U_{2l}, U_{3l})=(U_2, U_3)\;\;\;{\rm in}\;\;
 [C^2_{\rm loc}([0,\yy))]^2,
 \lbl{2.9a}\ees
where $(U_2, U_3)$ is the unique bounded positive solution of \qq{2.4}. Moreover, if $\rho'(x)\ge,\,\not\equiv 0$ in $[0,\yy)$, then $U_2',\, U_3'>0$ in $[0,\yy)$ and
\bes
\lim_{x\to\yy}(U_2, U_3)=
\big(b\beta\big(1-\sqrt{cq/(bk\beta)}\,\big)/c,\; \sqrt{bk\beta/(cq)}-1\big).
\lbl{2.9}\ees
\end{theo}

The proof of Theorem \ref{th2.2} will be divided into three lemmas.
In the first lemma, we first use the theory of topological degree in cones to prove the existence of positive solution $(U_{2l},  U_{3l})$ of \qq{2.5}, and then prove the uniqueness by the comparison arguments. In the second lemma, we first prove the limit \qq{2.9a} and prove that the limit function $(U_2, U_3)$ has a positive lower bound, and then prove the uniqueness of bounded positive solution of \qq{2.4} by the comparison arguments. In order to prove $U_2', U_3'>0$ in $[0,\yy)$ when $\rho'(x)\ge,\,\not\equiv 0$, in the third lemma we first construct the suitable boundary value problems in bounded intervals $(0,l)$ and then show that their unique positive solutions are increasing in $(0,l)$. Then we prove that the limit $(U_2, U_3)$ of these positive solutions is a bounded positive solution of \qq{2.4} and satisfies $U_2', U_3'>0$ in $[0,\yy)$.

Let
 \[U=(U_2, U_3),\;\;\mathscr{L}={\rm dig}( -d_2\Delta,-d_3\Delta),\;\;\;g(x,U)=(g_2(x,U),\,g_3(U)),\]
where $\Delta=\,''$. Then \qq{2.5} can be written as
 \bess\begin{cases}
 \mathscr{L}U=g(x,U), &0<x<l,\\
 U=(0,0),&x=0,\; l.
 \end{cases}\eess
Set
 \bess
  E&=&X^2 \;\;\text{with}\;\; X=\{z\in C^1([0,l]):\, z(0)=z(l)=0\},\\[1mm]
  W&=&K^2\;\;\text{with}\;\; K=\{z\in X:\, z\geq 0\}.
 \eess
For $U\in W$, we define
 \bess
 W_{U}&=&\{V\in E:\, \exists\; r>0\;\; \text{s.t.}\;\;U+tV\in
W,\,\;\forall\; 0\leq t\leq r\},\\
 S_{U}&=&\{V\in\overline{W}_{U}:-V\in\overline{W}_{U}\}.
 \eess

Let $0\le\tau\le 1$ and $U^\tau$ be a nonnegative solution of
\bes\begin{cases}
 \mathscr{L}U=\tau g(x,U), &0<x<l,\\
 U=0,&x=0,\; l.
 \end{cases}\lbl{2.10}\ees
Set $\beta_m=\sup_{[0,\yy)}\rho(x)$. Then we have $U_2^\tau(x)\le b\beta_m/c$ and $U_3^\tau(x)\leq bk\beta_m/(cq)$ in $[0, l]$ by the maximum principle, and $\|U^\tau\|_{C^1([0, l])}\le C$ for all $0\le\tau\le 1$ by the standard elliptic theory and imbedding theorem. Define
  \[\mathcal{O}=\{U\in W:\, \|U\|_{C^1([0,l])}<C+1\}.\]
Then the problem \qq{2.10} has no solution on $\partial{\mathcal O}$.
Take a large constant $M$ such that
 \[\tau g(x,U)+MU\ge(0,0),\;\;
 \forall\; x\in[0,l],\; U\in\ol{\mathcal{O}},\;0\le\tau\le 1,\]
and define operator $G_\tau$ by
 \[G_\tau(U)=(\mathscr{L}+M)^{-1}(\tau g(x,U)+MU),\;\;U\in E.\]
Then $G_\tau:\, \mathcal{O}\to W$ is positive and compact for any $0\le\tau\le 1$, and $U\in{\mathcal O}$ is a solution of \eqref{2.1} if and only if $G_1(U)=U$.
Moreover, $G_\tau(U)\not=U$ for any $0\leq\tau\leq 1$ and $U\in\partial{\mathcal O}$. Thus, by the homotopy invariance we have
 \bess
 {\rm deg}_W(I-G_\tau, {\mathcal O})={\rm deg}_W(I-G_0, {\mathcal O})=1,\;\;\forall \, 0\leq \tau\leq 1.
 \eess
The proof of ${\rm deg}_W(I-G_0, {\mathcal O})=1$ is standard (cf. \cite[Lemma 5.1]{WPbook24}). We denote $G_1=G$.

\begin{lem}\lbl{l2.1} Let $bk\beta/(cq)>1$ and $l^*$ be given in the above. Then the problem \qq{2.5} has a unique positive solution $(U_{2l},  U_{3l})$ when $l> 2l^*$.
\end{lem}

\begin{proof} It is clear that $(0, 0)$ is the unique trivial nonnegative solution of \qq{2.5}. Now we calculate ${\rm index}_W(G, (0,0))$. By the direct calculations,
 \[\overline{W}_{(0,0)}=W,\;\;S_{(0,0)}=\{(0,0)\},\;\;
 G'(0,0)=(\mathscr{L}+M)^{-1}\kk(\begin{array}{ccc}
  M-c \;\; &b \rho(x)\\
 k \; \; & M-q\end{array} \rr). \]

{\it Step 1}. We first show that the problem
 \bes
 G'(0,0)\phi=\phi,\;\;\phi\in \overline{W}_{(0,0)}=W
 \lbl{2.11}
 \ees
has only the zero solution when $l>2l^*$. Let $\phi=(\phi_2, \phi_3)\in W$ be a  solution of \qq{2.11}. Then $\phi_i\ge 0$ in $(0, l)$, $i=1,2$ and
 \bes\begin{cases}
  -d_2\phi_2''=-c\phi_2+b \rho(x)\phi_3, &0<x<l,\\
 -d_3\phi_3''=k\phi_2-q\phi_3, \;\;&0<x<l,\\
 \phi_i(0)=\phi_i(l)=0,\;\;&i=2, 3.
 \end{cases}\lbl{2.12}\ees
We claim that $(\phi_2, \phi_3)=(0,0)$.

Assume on the contrary that $\phi_3\not\equiv 0$, then $\phi_2\not\equiv 0$. So $\phi_3>0$ and $\phi_2>0$ in $(0,l)$ as $\rho\ge,\not\equiv 0$ in $[0,l]$. Let $\lm_1(l^*)$ be the principle eigenvalue of \qq{2.7} with $l=l^*$ and $\psi$ be the corresponding positive eigenfunction. Then $\psi'(l^*)>0$, $\psi'(2l^*)<0$. Noticing that $l>2l^*$, we have $\phi_i(l^*), \phi_i(2l^*)>0$ for $i=2,3$. Multiplying the first two equations of \qq{2.12} by $\psi$ and integrating the results over $[l^*, 2l^*]$ and using fact \qq{2.8}, we have that, by the direct calculation,
 \bess
d_2\lm_1(l^*)\int_{l^*}^{2l^*}\!\phi_2\psi\dx
 &>&d_2\big[\phi_2(2l^*)\psi'(2l^*)-\phi_2(l^*)\psi'(l^*)\big]
 +d_2\lm_1(l^*)\int_{l^*}^{2l^*}\!\phi_2\psi\dx\\
 &=&-c\int_{l^*}^{2l^*}\!\phi_2\psi\dx
 +b\int_{l^*}^{2l^*}\!\rho\phi_3\psi\dx\\
 &>&-c\int_{l^*}^{2l^*}\!\phi_2\psi\dx
 +b\kk(\beta-\ep\rr)\int_{l^*}^{2l^*}\!\phi_3\psi\dx,\\
 d_3\lm_1(l^*)\int_{l^*}^{2l^*}\!\phi_3\psi\dx&>&k\int_{l^*}^{2l^*}\!\phi_2\psi\dx
 -q\int_{l^*}^{2l^*}\!\phi_3\psi\dx.
  \eess
Thanks to $\int_{l^*}^{2l^*}\!\phi_i\psi\dx>0$ for $i=2,3$. It is not hard to derive from the above inequalities that $bk\kk(\beta-\ep\rr)<(c+d_2\lm_1(l^*))(q+d_3\lm_1(l^*))$. This contradicts to \qq{2.6}.

{\it Step 2}. Next we show that the eigenvalue problem
   \bes
 G'(0,0)\phi=\mu\phi,\;\;\phi\in\overline{W}_{(0,0)}\setminus S_{(0,0)}
 =W\setminus\{(0,0)\}
 \lbl{2.13}
 \ees
has an eigenvalue $\mu>1$. In fact, according to \cite[Theorem 1]{Hess83}, the eigenvalue problem
 \bess\begin{cases}
 -d_2\phi_2''+M\phi_2=r[(M-c)\phi_2+b \rho(x)\phi_3],\; &0<x<l,\\
 -d_3\phi_3''+M\phi_3=r[k\phi_2+(M-q)\phi_3],\; &0<x<l,\\
 \phi_i(0)=\phi_i(l)=0, \;\;i=2,3
 \end{cases}\eess
has a positive principle eigenvalue $r$ and with positive eigenfunction $(\phi_2, \phi_3)$, i.e., $\phi_2, \phi_3>0$. By the results of Step 1 we see that $r\not=1$. In the following we shall show that $r<1$.

Assume on the contrary that $r>1$. Since $M>c$ and $M>q$, it follows that
 \bess\begin{cases}
 -d_2\phi_2'=[r(M-c)-M]\phi_2+rb \rho(x)\phi_3]>-c\phi_2+b \rho(x)\phi_3,\; &0<x<l,\\
 -d_3\phi_3''=rk\phi_2+[r(M-q)-M]\phi_3]>k\phi_2-q\phi_3,\; &0<x<l,\\
 \phi_i(0)=\phi_i(l)=0, \;\;i=2,3.
 \end{cases}\eess
Similar to Step 1 we can derive a contradiction. Thus $r<1$.
Obviously, $\mu=1/r>1$ is an eigenvalue of \qq{2.13} since $(\phi_2, \phi_3)\in\overline{W}_{(0,0)}\setminus S_{(0,0)}$.

{\it Step 3}. The existence. Take advantage of \cite[Corollry 3.1]{WYcnsns24}, ${\rm index}_W(G,(0,0))=0$. It is followed by the theory of topological degree in cones that the operator $G$ has a fixed point within $\mathcal{O}$ that is different from $(0,0)$. Hence, the problem \qq{2.5} has a solution in $\mathcal{O}$ that is different from $(0,0)$. It is clear that such a solution is positive. So \qq{2.5} has least one positive solution.

{\it Step 4}. The uniqueness. Suppose that $(U_{2l}, U_{3l})$ is a positive solution of \qq{2.5}. By use of the maximum principle we can get
 \[U_{2l}<b\beta_m/c,\;\; U_{3l}<bk\beta_m/(cq)\;\;\;{\rm in}\;\;(0, l),\]
where $\beta_m=\sup_{[0,\yy)}\rho(x)$. Take $\ol U_{2l}=b\beta_m/c$ and $\ol U_{3l}=bk\beta_m/(cq)$. Then $(\ol U_{2l}, \ol U_{3l})\geq(U_{2l}, U_{3l})$ and $(\ol U_{2l}, \ol U_{3l})$ is an upper solution of \qq{2.5}. By the monotone iterative method, the problem \qq{2.5} has a positive solution $(\hat U_{2l},\hat U_{3l})$ and $(\hat U_{2l},\hat U_{3l})\geq(U_{2l}, U_{3l})$.

In view of $U_{2l}'(0), U_{3l}'(0)>0$ and $U_{2l}'(l), U_{3l}'(l)<0$, we can find $\xi>1$ such that $\xi(U_{2l}, U_{3l})\ge (\hat U_{2l}, \hat U_{3l})$ in $[0,l]$. Set
\[\ud\xi=\inf\{\xi\ge1: \xi(U_{2l}, U_{3l})\ge(\hat U_{2l}, \hat U_{3l}) {\rm ~ in  ~ }[0,l]\}.\]
It is clear that $\ud\xi$ is well defined, $\ud\xi\ge1$ and $\ud\xi(U_{2l}, U_{3l})\ge(\hat U_{2l}, \hat U_{3l})$ in $[0,l]$.

We shall prove $\ud\xi=1$. Assume $\ud\xi>1$. Let $\phi(x)=\ud\xi U_{2l}-\hat U_{2l}$ and $\psi(x)=\ud\xi U_{3l}-\hat U_{3l}$. In then follows that, by the carefully calculations,
 \bess
-d_2\phi''+c\phi&=&\frac{b \rho(x)\psi}{(1+U_{3l})(1+\hat U_{3l})}+\frac{b(\ud\xi-1) \rho(x)U_{3l}\hat U_{3l}}{(1+U_{3l})(1+\hat U_{3l})}, ~ ~\, x\in[0,l],\\[2mm]
-d_3\psi''+q\psi&=&\frac{k\phi}{1+U_{3l}}+\frac{k\hat U_{2l}(\hat U_{3l}-U_{3l})}{(1+U_{3l})(1+\hat U_{3l})}, \hspace{21mm} x\in[0,l],\\[1.5mm]
\phi(0)=\psi(0)&=&\phi(l)=\psi(l)=0.
 \eess
As $\hat U_{3l}-U_{3l}\ge 0$, $(\ud\xi-1)\rho U_{3l}\hat U_{3l}\ge,\,\not\equiv 0$ and $\phi, \psi\ge 0$, we have that $\phi, \psi> 0$ in $(0,l)$ by the maximum principle, and $\phi'(0), \psi'(0)>0$ and $\phi'(l),\psi'(l)<0$ by the Hopf boundary lemma. Thus, we can find $\ep>0$ such that $(\phi, \psi)\ge\ep(U_{2l}, U_{3l})$, i.e., $(\ud\xi-\ep)(U_{2l}, U_{3l})\ge(\hat U_{2l}, \hat U_{3l})$ in $[0,l]$. This is contradicts to the definition of $\ud\xi$. Then $\ud\xi=1$, i.e., $(\hat U_{2l}, \hat U_{3l})=(U_{2l}, U_{3l})$. The proof is complete. \end{proof}

\begin{lem}\lbl{l2.2} Let $(U_{2l},  U_{3l})$ be the unique positive solution of \qq{2.5}. Then
 \bes
 \lim_{l\to\yy}(U_{2l}, U_{3l})=(U_2, U_3)\;\;\;{\rm in}\;\;
 [C^2_{\rm loc}([0,\yy))]^2,
 \lbl{2.14}\ees
and $(U_2, U_3)$ is the unique bounded positive solution of \qq{2.4}.
\end{lem}

\begin{proof} {\it Step 1}. It is clear that $(U_{2l},  U_{3l})$ is increasing in $l$ by the comparison principle. Using classical elliptic regularity theory (uniformly estimate) and a diagonal procedure we can show that the limit \qq{2.14} exists, and $(U_2, U_3)$ is a bounded positive solution of \qq{2.4}.

{\it Step 2}. Prove that there exists $\tau>0$ such that
 \bes
 U_2(x), \;\;U_3(x)\ge\tau,\;\;\forall\; x\ge 1.
 \lbl{2.15}\ees

{\bf Claim 1}: $\lim_{x\to\yy}U_3(x)=0$ is impossible.\vskip 2pt

As $bk\beta/(cq)>1$, there exists $\ep>0$ such that
$\frac{bk(1-\ep)(\beta-\ep)}{c(1+\ep)^2}-q>0$. If $\lim_{x\to\yy}U_3(x)=0$, then there exists $x_0\gg1$ such that $\rho(x)\ge \beta-\ep$ and $U_3(x)<\ep$ for all $x\ge x_0$. Set
 \[d=\frac{k(1-\ep)}{c(1+\ep)},\;\;\;{\rm and}\;\; Z(x)=d_2dU_2(x)+d_3U_3(x).\]
Then $w$ satisfies, for some constant $\sigma>0$,
 \bess
 -Z''&=&\frac{db\rho(x)U_3}{1+U_3}-dc U_2+\frac{kU_2}{1+U_3}-qU_3\nonumber\\[1mm]
 &\ge&\frac{db(\beta-\ep)U_3}{1+\ep}-dc U_2+\frac{kU_2}{1+\ep}-qU_3\nonumber\\[1mm]
 &=&\frac{k\ep}{1+\ep}U_2+\kk(\frac{bk(1-\ep)(\beta-\ep)}
 {c(1+\ep)^2}-q\rr)U_3\nonumber\\[1mm]
 &\ge&\sigma Z,\;\;\;x>x_0.
 \eess
Noticing that $Z(x)>0$. It follows from above that $Z'(x)>0$ for $x>x_0$, and so $Z(x)\ge\tau_1>0$. As $\lim_{x\to\yy}U_3(x)=0$, it yields $U_2(x)\ge\tau_2>0$ and there exists $x_1\gg1$ such that
 \[-d_3U_3''=\dd\frac{kU_2}{1+U_3}-qU_3\ge\tau>0,\;\;\;x\ge x_1.\]
This is impossible since $U_3(x)>0$ for all $x>0$.

{\bf Claim 2}: $\lim_{x\to\yy}U_2(x)=0$ is impossible.\vspace{1mm}

If $\lim_{x\to\yy}U_2(x)=0$, it is easy to see from the first equation of \qq{2.4} that $U_3$ has no positive lower bound since $\lim_{x\to\yy}\rho(x)=\beta>0$. On the other hand, noticing that $\lim_{x\to\yy}U_3(x)=0$ is not true, there exist $\sigma>0$ and $x_n\to\yy$ such that $U_3$ reaches a local maximum at $x_n$ and $U_3(x_n)\ge\sigma$. Thus we have
 \[kU_2(x_n)\ge qU_3(x_n)(1+U_3(x_n))\ge\sigma(1+\sigma),\]
which is impossible as $\lim_{n\to\yy}U_2(x_n)=0$. \vskip 2pt

{\bf Claim 3}: Both $\liminf_{x\to\yy}U_3(x)=0$ and $\liminf_{x\to\yy}U_2(x)=0$ \vspace{1mm}are impossible.

If $\liminf_{x\to\yy}U_3(x)=0$, then, by {\bf Claim 1}, there exist $x_{n+2}>x_{n+1}>x_n\to\yy$ such that $U_3$ reaches local maximums  at $x_n$ and $x_{n+2}$, and local minimum at $x_{n+1}$; $U_3$ is decreasing in $(x_n, x_{n+1})$ and increasing in $(x_{n+1}, x_{n+2})$; and $\lim_{n\to\yy}U_3(x_{n+1})=0$. Thus,
 \bes
 qU_3(x_n)(1+U_3(x_n))\le kU_2(x_n),\;\;\;kU_2(x_{n+1})\le qU_3(x_{n+1})(1+U_3(x_{n+1}))\to 0,
 \lbl{2.17}\ees
which implies $\lim_{n\to\yy}U_2(x_{n+1})=0$.

We first show that $U_2$ cannot be monotonically increasing in the left neighborhood of $x_{n+1}$. If this is not true, as $\lim_{n\to\yy}U_2(x_{n+1})=0$, we can find the first  $\ud x_n<x_{n+1}$ such that $U_2$ reaches a local minimum at $\ud x_n$ and $U_2$ is increasing in $(\ud x_n, x_{n+1})$. Then
$\frac{b\rho(\ud x_n)U_3(\ud x_n)}{1+U_3(\ud x_n)}\le cU_2(\ud x_n)$,
and, by second inequality of \qq{2.17},
 \bes
 \frac{b\rho(\ud x_n)U_3(\ud x_n)}{1+U_3(\ud x_n)}\le cU_2(\ud x_n)
 \le cU_2(x_{n+1})\le\frac {cq}kU_3(x_{n+1})(1+U_3(x_{n+1})),
 \lbl{2.19a}\ees
which implies $\lim_{n\to\yy}U_3(\ud x_n)=0$ as $\lim_{n\to\yy}U_3(x_{n+1})=0$. Recall that $\lim_{n\to\yy}\rho(\ud x_n)=\beta$ and ${bk\beta}/(cq)>1$. It can be  derived from \qq{2.19a} that $U_3(\ud x_n)<U_3(x_{n+1})$ when $n$ is large enough. Noticing that $U_3$ is decreasing in $(x_n, x_{n+1})$. It follows that $\ud x_n<x_n$, and then $U_2$ is increasing in $(x_n, x_{n+1})$. Making use of \qq{2.17} we have
 \bess
 qU_3(x_n)(1+U_3(x_n))\le kU_2(x_n)<kU_2(x_{n+1})\le qU_3(x_{n+1})(1+U_3(x_{n+1})),
 \eess
which implies $U_3(x_n)<U_3(x_{n+1})$. This is impossible.

Similarly, we can show that $U_2$ cannot be monotonically decreasing in the right neighborhood of $x_{n+1}$.

The only possibility is that $U_2'(x_{n+1})=0$ and $U_2''(x_{n+1})\ge 0$. Then $\frac{b\rho(x_{n+1})U_3(x_{n+1})}{1+U_3(x_{n+1})}\le cU_2(x_{n+1})$,
and, by the second inequality of \qq{2.17},
 \bess
\frac{b\rho(x_{n+1})U_3(x_{n+1})}{1+U_3(x_{n+1})}\le cU_2(x_{n+1})\le\frac {cq}kU_3(x_{n+1})(1+U_3(x_{n+1})).
 \eess
Thus we have $bk\rho(x_{n+1})\le cq(1+U_3(x_{n+1}))^2$.
This is impossible since $\lim_{n\to\yy}\rho(x_{n+1})=\beta$, $\lim_{n\to\yy}U_3(x_{n+1})=0$ and $bk\beta>cq$.

Similarly, we can show that $\liminf_{x\to\yy}U_2(x)=0$ is impossible. {\bf Claim 3} is proved.

Summarizing, we have proved that \qq{2.15} is true.

{\it Step 3}. Prove the uniqueness. The proof is similar to Step 4 in the proof of Lemma \ref{l2.1}. For completeness and convenience of readers, we provide details.
It is clear that $U_{2l}'(0),  U_{3l}'(0)>0$ by the Hopf boundary lemma, and $(U_{2l},  U_{3l})$ is increasing in $l$ by the comparison principle. Consequently, $U_2'(0), U_3'(0)>0$.

Let $(\hat U_2,  \hat U_3)$ be a bounded positive solution of \qq{2.4}. Then $(\hat U_2,  \hat U_3)>(U_{2l},  U_{3l})$ in $(0,l)$ by the comparison principle, and so $(\hat U_2,  \hat U_3)\ge (U_2, U_3)$ in $(0, \yy)$. In view of $U_2'(0), U_3'(0)>0$ and \qq{2.15}, we can find $\xi>1$ such that $\xi(U_2, U_3)\ge (\hat U_2, \hat U_3)$ in $[0,\yy)$. Set
\[\ud\xi=\inf\{\xi\ge1: \xi(U_2, U_3)\ge(\hat U_2, \hat U_3) {\rm ~ in  ~ }[0,\yy)\}.\]
It is clear that $\ud\xi$ is well defined, $\ud\xi\ge1$ and $\ud\xi(U_2, U_3)\ge(\hat U_2, \hat U_3)$ in $[0,\yy)$.

We shall prove $\ud\xi=1$. Assume $\ud\xi>1$. Let $\phi(x)=\ud\xi U_2-\hat U_2$ and $\psi(x)=\ud\xi U_3-\hat U_3$. It then follows that, by the carefully calculations,
 \bes
-d_2\phi''+c\phi&=&\frac{b\rho(x)\psi}{(1+U_3)(1+\hat U_3)}+\frac{b(\ud\xi-1)\rho(x)\hat U_3U_3}{(1+U_3)(1+\hat U_3)}, ~ ~\, x\in[0,\yy),\lbl{2.18}\\[1mm]
-d_3\psi''+q\psi&=&\frac{k\phi}{1+U_3}+\frac{kv^*(\hat U_3-U_3)}{(1+U_3)(1+\hat U_3)}, \hspace{21mm} x\in[0,\yy).\nonumber
 \ees
As $\hat U_3-U_3\ge 0$, $(\ud\xi-1)\rho \hat U_3U_3\ge,\,\not\equiv 0$ and $\phi, \psi\ge 0$, we have $\phi, \psi> 0$ in $[0,\yy)$. Thus, by \qq{2.18},
 \bes
 -d_2\phi''+c\phi\ge\frac{b(\ud\xi-1)\rho(x)U_3\hat U_3}{(1+U_3)(1+\hat U_3)},\;\;x\in[0,\yy).
 \lbl{2.19}\ees
In view of \qq{2.15}, we have $(\hat U_2,  \hat U_3)\ge(U_2, U_3)\ge(\tau, \tau)$ for $x\ge 1$. Noticing that $\rho(x)\ge 0$ and $\lim_{x\to\yy}\rho(x)=\beta$. It follows from \qq{2.19} that there exists $\delta>0$ such that $\liminf_{x\to\yy}\phi(x)\ge\delta$. Then, by use of \qq{2.18}, there is $\sigma>0$ such that $\liminf_{x\to\yy}\psi(x)\ge\sigma$. So, we can find $\ep>0$ such that $(\phi, \psi)\ge\ep(U_2, U_3)$ in $[0,\yy)$. This contradicts to the definition of $\ud\xi$. Hence, $\ud\xi=1$, i.e., $(\hat U_2, \hat U_3)=(U_2, U_3)$. The proof is complete.
\end{proof}

\begin{lem}\lbl{l2.3} If $\rho'(x)\ge,\,\not\equiv 0$ in $[0,\yy)$, then the unique bounded positive solution $(U_2,  U_3)$ of \qq{2.4} satisfies $U_2'(x)>0,  U_3'(x)>0$ in $[0,\yy)$ and \qq{2.9} holds.
\end{lem}

\begin{proof} {\it Step 1}. Take $\zeta>1$ large enough such that
 \bes
 \frac{b\beta}{c(1+\zeta)}<1, \;\;\;\frac{k}{q(1+\zeta)}<1.
 \lbl{2.20}\ees

Let $d, \gamma>0$ be constants and $z\in C([0,l])$ is increasing and $|z|\le M\zeta$ for some positive constant $M$. Let $0\le u\le\zeta$ be the unique solution of
 \bess\bbb
 -du''+\gamma u=z(x),\;\;\;0<x<l,\\
 u(0)=0,\;\;u(l)=\zeta.
 \nnn\eess
Making use of the $L^p$ theory we have that there exists a constant $C(d,\gamma, M)>0$ such that $\|u\|_{C^{1+\alpha}([0,l])}\le C(d,\gamma, M)$. We define
 \bess
 {\sum}_2&=&\{U_2\in C^1([0,l]):\, \|U_2\|_{C^1([0,l])}\le C(d_2,c, b\beta),\; 0\le U_2\le\zeta,\;U_2'\ge 0\},\\
 {\sum}_3&=&\{U_3\in C^1([0,l]):\, \|U_3\|_{C^1([0,l])}\le C(d_3,q, k),\; 0\le U_3\le\zeta,\;U_3'\ge 0\}.
 \eess
Then ${\sum}_3$ is a bounded and closed convex set of $C^1([0,l])$.

For the given $U_3\in{\sum}_3$. Let $U_2$ be the unique positive solution of \bess\begin{cases}
-d_2U_2''+cU_2=\dd\frac{b\rho(x)U_3(x)}{1+U_3(x)}, \;\; 0<x<l,\\[2mm]
U_2(0)=0,\;\;U_2(l)=\zeta.
  \end{cases}\eess
If $U_2$ takes a local maximum at $x_0\in(0,l)$, then
 \[U_2(x_0)\le\frac{b\rho(x_0)U_3(x_0)}{c(1+U_3(x_0))}< \frac{b\beta\zeta}{c(1+\zeta)}<\zeta\]
by the first inequality of \qq{2.20}. Consequently, $U_2(x)\le\zeta$ in $[0,l]$. Moreover, it is easy to see that the function $z(x)=\frac{b\rho(x)U_3(x)}{1+U_3(x)}$ is increasing in $x$ and $0\le z<b\beta\zeta$. Thus $U_2(x)$ is increasing and $\|U_2\|_{C^{1+\alpha}([0,l])}\le C(d_2,c,b\beta)$. This indicates that $U_2\in{\sum}_2$.

For such a $U_2$, let $\hat U_3$ be the unique positive solution of \bess\begin{cases}
-d_3\hat U_3''+q\hat U_3=\dd\frac{k U_2(x)}{1+\hat U_3(x)}, \;\; 0<x<l,\\[2mm]
\hat U_3(0)=0,\;\;\hat U_3(l)=\zeta.
  \end{cases}\eess
If $\hat U_3$ takes a local maximum at $x_0\in(0,l)$ and $\hat U_3(x_0)\ge\zeta$, then
 \[q\zeta(1+\zeta)\leq\hat U_3(x_0)(1+\hat U_3(x_0))\le kU_2(x_0)\le k\zeta.\]
This contradicts the second inequality of \qq{2.20}. Thus, $\hat U_3\le\zeta$, and so $\|\hat U_3\|_{C^{1+\alpha}([0,l])}\le C(d_3,q,k)$. As $U_2(x)$ is increasing, we have that $\hat U_3(x)$ is increasing. Therefore, $\hat U_3(x)\in{\sum}_3$.

Define ${\mathscr F}(U_3)=\hat U_3$. The above arguments show that ${\mathscr F}:\,{\sum}_3\to{\sum}_3$. It is clear that ${\mathscr F}$ is continuous based on the continuous dependence of the solution on parameters. This combines the estimate $\|F(U_3)\|_{C^{1+\alpha}([0,l])}\le C(d_3,q,k)$ indicate that ${\mathscr F}$ is compact. Hence, by the Schauder fixed point theorem, ${\mathscr F}$ has at least one fixed point $U_3\in{\sum}_3$. Thus, the boundary value problem
 \bes\begin{cases}
-d_2U_2''=g_2(x,U_2,U_3),\;\;\;\;0<x<l,\\
-d_3U_3''=g_3(U_2,U_3), \;\,\;\;\;\;\;\;0<x<l,\\
U_2(0)=U_3(0)=0, \;\;U_2(l)=U_3(l)=\zeta
  \end{cases}\label{2.21}\ees
has at least one positive solution $(U_{2l}^\zeta, U_{3l}^\zeta)$, and $U_{2l}^\zeta(x)$ and $U_{3l}^\zeta(x)$ are increasing in $x$.
Obviously, $(U_{2l}^\zeta, U_{3l}^\zeta)\ge(U_{2l},  U_{3l})$ when $l>2l^*$, where $(U_{2l},  U_{3l})$ is the unique positive solution of \qq{2.5}

{\it Step 2}. Let $(U_2, U_3)$ be a positive solution of \qq{2.21}. Then
$U_2, U_3\le\zeta$ by the maximum principle. Thanks to \qq{2.20}. It is easy to see that $(U_2, U_3)=(\zeta, \zeta)$ and $(\ud U_2, \ud U_3)=(0,0)$ are the ordered upper and lower solutions of \qq{2.21}. By the upper and lower solutions method (the monotone iterative method), the problem \qq{2.21} has at least one positive solution $(\hat U_2, \hat U_3)$, and $(\hat U_2, \hat U_3)$ is the maximal solution  of \eqref{2.21} located between $(0,0)$ and $(\zeta, \zeta)$. Thus, $(U_2, U_3)\le(\hat U_2, \hat U_3)$.

Similar to Step 4 in the proof of Lemma \ref{l2.1} we can show that $(U_2, U_3)=(\hat U_2, \hat U_3)$. The uniqueness is obtained.

{\it Step 3}. Using classical elliptic regularity theory (uniformly estimate) and a diagonal procedure we can show that there exists a subsequence of $\{(U_{2l}^\zeta, U_{3l}^\zeta)\}$, denoted by it self, and $U^*_2, U^*_3\in C^2([0,\yy))$ such that $(U_{2l}^\zeta, U_{3l}^\zeta)\to (U^*_2, U^*_3)$ in $[C^2_{\rm loc}([0,\yy))]^2$ as $l\to\yy$, and $(U^*_2, U^*_3)$ is a bounded nonnegative solution of \qq{2.4}. Moreover, $U^*_2$ and $U^*_3$ are increasing in $x\in[0,\yy)$ as $U_{2l}^\zeta(x)$ and $U_{3l}^\zeta(x)$ are increasing in $x\in[0,l]$, and $(U^*_2, U^*_3)\ge (U_2, U_3)$ since $(U_{2l}^\zeta, U_{3l}^\zeta)\ge(U_{2l},  U_{3l})$. Hence, $(U^*_2, U^*_3)$ is a bounded positive solution of \qq{2.4}. By the uniqueness of bounded positive solutions of \qq{2.4} (Lemma \ref{l2.2}), $(U^*_2, U^*_3)=(U_2, U_3)$. So, $U_2$ and $U_3$ are increasing in $x\in[0,\yy)$. Therefore, the limit \qq{2.9} exists.

Let $U_2'=V$ and $p(x)=\frac{b\rho(x)U_3}{1+U_3}$. Then $p'(x)\ge,\,\not\equiv 0$ in $[0,\yy)$. There exists $x_0>0$ such that $p'(x)\ge,\,\not\equiv 0$ in $[x_0,\yy)$. For any given $l>x_0$, as $V$ satisfies
 \bess\begin{cases}
 -d_2V''+cV=p'(x)\ge,\,\not\equiv0,&0<x<l,\\
 V(0)>0,\;\; V(l)\ge 0,
 \end{cases}\eess
it derives that $V>0$ in $(0,l)$ by the maximum principle. The arbitrariness of $l>x_0$ shows that $V>0$, i.e., $U_2'>0$ in $(0,\yy)$. By the same way, we can show that $U_3'>0$ in $(0,\yy)$.
\end{proof}
Clearly, Theorem \ref{th2.2} follows from Lemmas \ref{l2.1}-\ref{l2.3}.

\begin{theo}\lbl{th2.3} Let $(U_2, U_3)$ be the unique bounded positive solution of \qq{2.4} obtained in Theorem \ref{th2.2}. Then the problem
 \bes\begin{cases}
-d_1U_1''=f_1(U_1, U_3(x)), \; &0<x<\yy,\\
U_1=0, & x=0
  \label{2.22}\end{cases}\ees
has a unique bounded positive solution $U_1$. Moreover, if $\rho'(x)\ge,\,\not\equiv 0$ in $[0,\yy)$, then either $U_1(x)$ is increasing in $[0,\yy)$, or there exists $x_0>0$ such that $U_1(x)$ is increasing in $(0,x_0)$ and decreasing in $(x_0, \infty)$. Therefore,
  \[\lim\limits_{x\to\infty}U_1(x)=\frac{\theta\sqrt{bk\beta/(cq)}}
  {(a+b)\sqrt{bk\beta/(cq)}-b}.\]
\end{theo}

\begin{proof} The proof is similar to that of \cite[Theorem 2.1]{Wjde14},  and the details are omitted here.
\end{proof}

Let $U_1(x)$ be the unique positive solution of \qq{2.22} obtained in Theorem \ref{th2.3}. Taking $\rho(x)=U_1(x)$ in \qq{2.4} and using Theorem \ref{th2.2} we have the following results.

\begin{theo}\lbl{th2.4} If we further assume that $\rho'(x)\ge,\,\not\equiv 0$ in $[0,\yy)$, and
 \bess
 \frac{bk\theta\sqrt{bk\beta/(cq)}}
  {(a+b)\sqrt{bk\beta/(cq)}-b}>cq,
  \eess
then the problem
 \bes\begin{cases}
-d_2U_2''=f_2(U_1(x), U_2, U_3), &0<x<\yy,\\
-d_3U_3''=f_3(U_2, U_3), &0<x<\yy,\\
U_2=U_3=0, & x=0
  \end{cases}\label{2.23}\ees
has a unique bounded positive solution $(U_2, U_3)$. Moreover, if $U_1$ is increasing in $[0,\yy)$, so are $U_2$ and $U_3$. Unfortunately, we can not to obtain the monotonicity of $(U_2, U_3)$ when $U_1(x)$ is increasing in $(0,x_0)$ and decreasing in $(x_0,\infty)$.
\end{theo}

Making use of Theorems \ref{th2.2}-\ref{th2.4} we have the following conclusions:
 \begin{enumerate}[$(1)$]
 \item\, Take $\beta=\theta/a$ and $\rho(x)=\ol U_1(x)$. If $\mathcal{R}_0={kb\theta}/(acq)>1$ then the problem
\bes\begin{cases}
-d_2\ol U_2''=f_2(\ol U_1(x), \ol U_2, \ol U_3), &0<x<\yy,\\
-d_3\ol U_3''=f_3(\ol U_2, \ol U_3), &0<x<\yy,\\
\ol U_2=\ol U_3=0, & x=0
  \end{cases}\label{2.24}\ees
has a unique bounded positive solution $(\ol U_2, \ol U_3)$. Moreover, $\ol U_2', \ol U_3'>0$ in $[0,\yy)$ and
 \bess
\lim_{x\to\yy}(\ol U_2, \ol U_3)=
\kk(\frac{b\theta(\sqrt{\mathcal{R}_0}-1)}{ac\sqrt{\mathcal{R}_0}}, \sqrt{\mathcal{R}_0}-1\rr).
 \eess
\item\, The problem
 \bes\begin{cases}
-d_1\ud U_1''=f_1(\ud U_1, \ol U_3(x)), \; &0<x<\yy,\\
\ud U_1=0, & x=0
  \label{2.25}\end{cases}\ees
has a unique bounded positive solution $\ud U_1$. Moreover, either $\ud U_1(x)$ is increasing in $[0,\yy)$, or there exists $x_0>0$ such that $\ud U_1(x)$ is increasing in $(0,x_0)$ and decreasing in $(x_0, \infty)$. Therefore,
  \[\lim\limits_{x\to\infty}\ud U_1(x)=\frac{\theta\sqrt{\mathcal{R}_0}}{(a+b)\sqrt{\mathcal{R}_0}-b}.\]
\item\, Assume further that
 \bes
 \mathcal{R}_0+\sqrt{\mathcal{R}_0}>b/a,
 \lbl{2.26}\ees
then the problem
 \bes\begin{cases}
-d_2\ud U_2''=f_2(\ud U_1(x), \ud U_2, \ud U_3), &0<x<\yy,\\
-d_3\ud U_3''=f_3(\ud U_2, \ud U_3), &0<x<\yy,\\
\ud U_2=\ud U_3=0, & x=0
  \end{cases}\label{2.27}\ees
has a unique bounded positive solution $(\ud U_2, \ud U_3)$. Moreover, if $\ud U_1$ is increasing in $[0,\yy)$, so are $\ud U_2$ and $\ud U_3$. Unfortunately, we can not to obtain the monotonicity $(\ud U_2, \ud U_3)$ when $\ud U_1(x)$ is increasing in $(0,x_0)$ and decreasing in $(x_0,\infty)$.
\end{enumerate}

\subsection{Existence of bounded positive solutions of \qq{2.1}}

Throughout this subsection we always assume that $\mathcal{R}_0={bk\theta }/(acq)>1$, i.e., $bk\theta >acq$. We shall prove the existence of bounded positive solutions of \qq{2.1}. We first use the upper and lower solutions method to deal with this topic under the stronger condition \qq{2.26}, and then use the theory of topological degree in cones to deal with it under the weaker condition $bk\theta>acq$.

\begin{theo}\lbl{th2.5} Under the condition \qq{2.26}, problem \qq{2.1} has at least one bounded positive solution. Moreover, any bounded positive solution $(U_1, U_2, U_3)$ of \qq{2.1} satisfies
 \bes
 (\ol U_1, \ol U_2, \ol U_3)\le(U_1, U_2, U_3)\le (\ud U_1, \ud U_2, \ud U_3), \;\; \forall\; x\ge 0,\lbl{2.28}\ees
where $\ol U_1$, $(\ol U_2, \ol U_3)$, $\ud U_1$ and $(\ud U_2, \ud U_3)$ are the unique bounded positive solutions of \qq{2.2}, \qq{2.24}, \qq{2.25} and \qq{2.27}, respectively.
\end{theo}

\begin{proof} It is easy to see that $(\ol U_1, \ol U_2, \ol U_3)$ and $(\ud U_1, \ud U_2, \ud U_3)$ are the coupled ordered upper and lower solutions of \qq{2.1}.

Let $l^*$ be given by \qq{2.8} with $\rho(x)=\ol U_1(x)$. For any given $l>2l^*$, we consider the problem
\bes\begin{cases}
-d_1U_1''=f_1(U_1,U_3), \; &0<x<l,\\
-d_2U_2''=f_2(U_1,U_2,U_3), &0<x<l,\\
-d_3U_3''=f_3(U_2,U_3), &0<x<l,\\
U_i(0)=0, \;\; U_i(l)=\ol U_i(l),\;\;&i=1,\,2,\,3.
  \label{2.29}\end{cases}\ees
It is obvious that $(\ol U_1, \ol U_2, \ol U_3)$ and $(\ud U_1, \ud U_2, \ud U_3)$ are also the coupled ordered upper and lower solutions of \qq{2.29}. By the standard upper and lower solutions method we know that \qq{2.29} has at least one positive solution, denoted by $(U_{1l}, U_{2l}, U_{3l})$, and
 \bes
 (\ol U_1, \ol U_2, \ol U_3)\le(U_{1l}, U_{2l}, U_{3l})\le (\ud U_1, \ud U_2, \ud U_3), \;\; \forall\; 0\leq x\leq l.\lbl{2.32a}\ees
Applying the local estimation and compactness argument, it can be concluded that there exist $U_1, U_2, U_3\in C^2([0,\yy))$, such that $\lim_{l\to\yy}(U_{1l}, U_{2l}, U_{3l})=(U_1, U_2, U_3)$ in $\left[C^2_{\rm loc}([0,\infty))\right]^3$, and $(U_1, U_2, U_3)$ is a bounded positive solution of \eqref{2.1}. From \qq{2.32a} we see that \qq{2.28} holds.
\end{proof}

In the following we will use the theory of topological degree in cones to prove the existence of bounded positive solutions of \qq{2.1} under the weaker condition $bk\theta>acq$. Some contents of this subsection are similar to the above subsection. For the convenience of readers and completeness, we will provide details.

For any given $l>0$, we first consider the following boundary value problem
\bes\begin{cases}
-d_1U_1''=f_1(U_1,U_3), \; &0<x<l,\\
-d_2U_2''=f_2(U_1,U_2,U_3), &0<x<l,\\
-d_3U_3''=f_3(U_2,U_3), &0<x<l,\\
U_1=U_2=U_3=0, & x=0,\, l.
  \label{2.30}\end{cases}\ees

As $bk\theta>acq$, we can find $\ep>0$ such that
 \bess
 bk\kk(\theta/a-\ep\rr)>(c+\ep d_2)(q+\ep d_3).\eess
It is easy to see that the boundary value problem
 \bess\begin{cases}
 -d_2U_1''=\theta -aU_1,\;\;0<x<l,\\
 U_1(0)=U_1(l)=0
 \end{cases}\eess
has a unique positive solution $\widetilde U_{1l}(x)$. Moreover,
$\lim_{l\to\yy}\widetilde U_{1l}(x)=\ol U_1(x)$ in $C^2_{\rm loc}([0,\yy))$,
where $\ol U_1(x)$ is the unique bounded positive solution of \qq{2.2} and
$\lim_{x\to\yy}\ol U_1(x)=\theta/a$. Let $\lm_1(l)>0$ be the principle eigenvalue of \qq{2.7}. Then there exist $l^*\gg 1$ and $l_0>2l^*$ such that
 \bes
 \lm_1(l)<\ep,\;\;\forall\, l\ge l^*;\;\;\;\widetilde U_{1l}(x)>\theta/a-\ep \;\;{\rm in}\;\; [l^*, 2l^*],\;\;\forall\;l\ge l_0.\lbl{2.31}
 \ees

\begin{theo}\lbl{th2.6} Let $\mathcal{R}_0={bk\theta}/(acq)>1$ and $l_0$ be given in the above. Then the problem \qq{2.30} has at least one positive solution $(U_{1l}, U_{2l}, U_{3l})$ when $l\ge l_0$. Moreover, there exists a subsequence of $\{(U_{1l}, U_{2l}, U_{3l})\}$, denoted by it self, such that
 \bes
 \lim_{l\to\yy}(U_{1l}, U_{2l}, U_{3l})=(U_1, U_2, U_3)
 \;\;\;{\rm in}\;\; [C^2_{\rm loc}([0,\yy))]^3,
 \lbl{2.34}\ees
and $(U_1, U_2, U_3)$ is a bounded positive solution of \qq{2.1}.
\end{theo}

\begin{proof} {\it Step 1}. Similar to the subsection \ref{s2.1}, we set $\boldsymbol{0}=(0,0,0)$ and
 \[U=(U_1, U_2, U_3),\;\;\mathscr{L}={\rm dig}(-d_1\Delta, -d_2\Delta,-d_3\Delta),\;\;\;f(U)=(f_1(U),\,f_2(U),\,f_3(U)),\]
where $\Delta=\,''$. And define
 \bess
  E&=&X^3 \;\;\text{with}\;\; X=\{z\in C^1([0,l]):\, z(0)=z(l)=0\},\\[1mm]
  W&=&K^3\;\;\text{with}\;\; K=\{z\in X:\, z\geq 0\}.
 \eess
For $U\in W$, we define $W_{U}$ and $S_{U}$ as the manner in subsection \ref{s2.1}.

Then \qq{2.30} can be written as
 \bes\begin{cases}
 \mathscr{L}(U)=f(U),\; &0<x<l,\\
 U=\boldsymbol{0},&x=0,\; l.
 \end{cases}\lbl{2.35}\ees
Let $0\le\tau\le 1$ and $U^\tau$ be a nonnegative solution of
\bes\begin{cases}
 \mathscr{L}(U)=\tau f(U),\; &0<x<l,\\
 U=\boldsymbol{0},&x=0,\; l.
 \end{cases}\lbl{2.36}\ees
It then follows that
 \bes
 U_1^\tau(x)\le\frac\theta a,\;\;\;U_2^\tau(x)\le \frac{b\theta(\sqrt{\mathcal{R}_0}-1)}{ac\sqrt{\mathcal{R}_0}},\;\;\;U_3^\tau(x)\leq\sqrt{\mathcal{R}_0}-1,\;\;\; 0\le x\le l
 \lbl{2.37}\ees
by the maximum principle, where $\mathcal{R}_0={kb\theta }/(acq)$. Moreover,
$\|U^\tau\|_{C^1([0, l])}\le C$ for all $0\le\tau\le 1$ by the standard elliptic theory and imbedding theorem. Define
  \[\mathcal{O}=\{U\in W:\, \|U\|_{C^1([0,l])}<C+1\}.\]
Then the problem \qq{2.36} has no solution on $\partial{\mathcal O}$.

Take a large constant $M$ such that
  \[\tau f(U)+MU\ge\boldsymbol{0},\;\; \forall\; U\in\ol{\mathcal{O}},\;0\le\tau\le 1.\]
And define operator $F_\tau$ by
 \[F_\tau(U)=(\mathscr{L}+M)^{-1}(\tau f(U)+MU),\;\;U\in E,\;0\le\tau\le 1.\]
Then $F_\tau:\, \mathcal{O}\to W$ is positive and compact for any given  $0\le\tau\le 1$, and $U\in{\mathcal O}$ is a solution of \eqref{2.35} if and only if $F_1(U)=U$. Moreover, $F_\tau(U)\not=U$ for any $0\leq\tau\leq 1$ and $U\in\partial{\mathcal O}$. Thus, by the homotopy invariance we have
 \bes
 {\rm deg}_W(I-F_\tau, {\mathcal O})={\rm deg}_W(I-F_0, {\mathcal O})=1,\;\;\forall \, 0\leq \tau\leq 1.
 \lbl{2.38}\ees
The proof of ${\rm deg}_W(I-F_0, {\mathcal O})=1$ is standard (cf. \cite[Lemma 5.1]{WPbook24}). We denote $F_1$ by $F$.

{\it Step 2}. Clearly, $(\widetilde U_{1l}, 0, 0)$ is the unique trivial nonnegative solution of \qq{2.30}. Now we calculate ${\rm index}_W(F, (\widetilde U_{1l}, 0,0))$. By the direct calculations,
 \[\overline{W}_{(\widetilde U_{1l},0,0)}=X\times K\times K,\;\;S_{(\widetilde U_{1l},0,0)}=X\times\{(0,0)\},\]
and
 \[F'(\widetilde U_{1l},0,0)=(\mathscr{L}+M)^{-1}\kk(\begin{array}{ccc}
  M-a \; &0 \; &-b\widetilde U_{1l}\\
  0 \; &\; M-c \; &b\widetilde U_{1l}\\
  0 \; &k \; \; & M-q\end{array} \rr). \]

{\it Step 3}. Prove that
 \bess
 F'(\widetilde U_{1l}, 0,0)\phi=\phi,\;\;\phi\in \overline{W}_{(\widetilde U_{1l}, 0,0)}
 \eess
has only the zero solution when $l>l_0$. Let $\phi=(\phi_1, \phi_2, \phi_3)\in \overline{W}_{(\widetilde U_{1l}, 0,0)}$ and satisfy $F'(\widetilde U_{1l}, 0,0)\phi=\phi$. Then $\phi_2, \phi_3\ge 0$ and $\phi$ satisfies
 \bess\begin{cases}
 -d_1\phi_1''=-a\phi_1-b\widetilde U_{1l}\phi_3, \;\;&0<x<l,\\
 -d_2\phi_2''=-c\phi_2+b\widetilde U_{1l}\phi_3, \;\;&0<x<l,\\
 -d_3\phi_3''=k\phi_2-q\phi_3, \;\;&0<x<l,\\
 \phi_i(0)=\phi_i(l)=0,\;\;&i=1, 2, 3.
 \end{cases}\eess
Making use of \qq{2.31}, similar to Step 1 in the proof of Lemma \ref{l2.1} we can show that $(\phi_2, \phi_3)=(0,0)$ and then $\phi_1=0$.

{\it Step 4}. We show that the eigenvalue problem
   \bes
 F'(\widetilde U_{1l}, 0,0)\phi=\mu\phi,\;\;\phi\in\overline{W}_{(\widetilde U_{1l}, 0,0)}\setminus S_{(\widetilde U_{1l}, 0,0)}
 \lbl{2.39}
 \ees
has an eigenvalue $\mu>1$. Owing to \cite[Theorem 1]{Hess83}, the eigenvalue problem
 \bess\begin{cases}
 -d_2\phi_2''+M\phi_2=r[(M-c)\phi_2+b\widetilde U_{1l}\phi_3],\; &0<x<l,\\
 -d_3\phi_3''+M\phi_3=r[k\phi_2+(M-q)\phi_3],\; &0<x<l,\\
 \phi_i(0)=\phi_i(l)=0, \;\;i=2,3
 \end{cases}\eess
has a positive principle eigenvalue $r$ with positive eigenfunction $(\phi_2, \phi_3)$, i.e., $\phi_2, \phi_3>0$. By the result of Step 3 we have $r\not=1$. In the following we shall show that $r<1$.

Assume on the contrary that $r>1$. Since $M>c$ and $M>q$, it follows that
 \bess\begin{cases}
 -d_2\phi_2'=[r(M-c)-M]\phi_2+rb\widetilde U_{1l}\phi_3]>-c\phi_2+b\widetilde U_{1l}\phi_3,\; &0<x<l,\\
 -d_3\phi_3''=rk\phi_2+[r(M-q)-M]\phi_3]>k\phi_2-q\phi_3,\; &0<x<l,\\
 \phi_i(0)=\phi_i(l)=0, \;\;i=2,3.
 \end{cases}\eess
Similar to Step 1 in the proof of Lemma \ref{l2.1} we can derive a contradiction. Thus $r<1$.

It is well known that the boundary problem
 \bess\begin{cases}
 -d_1\phi_1''+[M-r(M-a)]\phi_1=-rb\widetilde U_{1l}\phi_3,\; 0<x<l,\\
 \phi_1(0)=\phi_1(l)=0
  \end{cases}\eess
has a unique solution $\phi_1\in X$ as $M-r(M-a)>0$. Set $\phi=(\phi_1,\phi_2, \phi_3)$. Then $(\phi_1,\phi_2, \phi_3)\in\overline{W}_{(\widetilde U_{1l}, 0,0)}\setminus S_{(\widetilde U_{1l}, 0,0)}$ and $(r, \phi)$ satisfies
 \bess\begin{cases}
-d_1\phi_1''+M\phi_1=r[(M-a)\phi_1-b\widetilde U_{1l}\phi_3],\; &0<x<l,\\
 -d_2\phi_2''+M\phi_2=r[(M-c)\phi_2+b\widetilde U_{1l}\phi_3],\; &0<x<l,\\
 -d_3\phi_3''+M\phi_3=r[k\phi_2+(M-q)\phi_3],\; &0<x<l,\\
 \phi_i(0)=\phi_i(l)=0, \;\;&i=1,2,3
 \end{cases}\eess
This shows that $\mu=1/r>1$ is an eigenvalue of \qq{2.39}.

{\it Step 5}. Making use of \cite[Corollry 3.1]{WYcnsns24} we have ${\rm index}_W(F,(\widetilde U_{1l}, 0,0))=0$. Owing to \qq{2.38}, it is deduced by the theory of topological degree in cones that \qq{2.35} also has a solution within $\mathcal{O}$ that is different from $(\widetilde U_{1l}, 0,0)$. It is clear that  such a solution is positive. So \qq{2.30} has least one positive solution $(U_{1l}, U_{2l}, U_{3l})$ when $l\ge l_0$.

The estimates \qq{2.37} show that $(U_{1l}, U_{2l}, U_{3l})$ is bounded uniformly in $l>0$. Using classical elliptic regularity theory (uniformly estimate) and a diagonal procedure we can find a subsequence of $\{(U_{1l}, U_{2l}, U_{3l})\}$, denoted by it self, and $U_1, U_2, U_3
\in C^2([0,\yy))$ such that \qq{2.34} holds. It is obvious that $(U_1, U_2, U_3)$ is a bounded nonnegative solution of \qq{2.1}.

{\it Step 6}. We shall show that $U_1, U_2, U_3$ are positive in $(0,\yy)$. Firstly, since $U_1$ satisfies
 \bess
 -d_1U_1''+\kk(a+\frac{bU_3}{1+U_3}\rr)U_1=\theta>0,\;\;0<x<\yy,
 \eess
it is obvious that $U_1>0$ in $(0,\yy)$. If $U_2\not\equiv 0$, then $U_3$ satisfies
 \bess
 -d_3U_3''+qU_3=\dd\frac{kU_2}{1+U_3}\geq,\,\not\equiv 0,\;\;0<x<\yy.
 \eess
Therefore, $U_3>0$ and so $U_2>0$ in $(0,\yy)$ by the maximum principle. This indicates that if our conclusion is not true, then $U_2, U_3\equiv 0$ and $U_1=\ol U_1$, which is the unique bounded positive solution of \qq{2.2}. Hence $U_1$ is strictly increasing in $(0,\yy)$ and $\lim_{x\to\yy}U_1(x)=\theta/a$. That is
 \bess
 \lim_{l\to\yy}U_{1l}=U_1,\;\;\lim_{l\to\yy}U_{2l}=\lim_{l\to\yy}U_{3l}=0\;\;\;
 {\rm in}\;\;C^2_{\rm loc}([0,\yy)).
 \eess

Since $bk\theta>acq$, there exists $0<\ep\ll 1$ such that
 \bes
 bk(\theta-\ep)>acq.
 \lbl{2.40}\ees
Thanks to $\lim_{x\to\yy}U_1(x)=\theta/a$, there exists $L>0$ such that
$U_1(x)>(\theta-\ep/2)/a$ for $x\ge L$. Taking $L<\ell_n\nearrow\yy$. For each of these $\ell_n$, as $\lim_{l\to\yy}U_{1l}(x)=U_1(x)$, there exists $l_n>\ell_n$ such that
 \bess
 U_{1l_n}(x)>(\theta-\ep)/a=:m\;\;\;{\rm in}\;\;[L, \ell_n].
 \eess
Then $(U_{2l_n}, U_{3l_n})$ satisfies
\bess\begin{cases}
-d_2U_{2l_n}''\ge f_2(m, U_{2l_n}, U_{3l_n}), &L<x<\ell_n,\\
-d_3U_{3l_n}''=f_3(U_{2l_n}, U_{3l_n}), &L<x<\ell_n,\\
U_{2l_n}, U_{3l_n}>0, & x=L, \ell_n.
  \end{cases}\eess
Noticing that \qq{2.40} and according to Lemma \ref{l2.1}, there exists $l_*\gg1$ such that, when $\ell_n>l_*$, the problem
 \bess\begin{cases}
-d_2U_2''=f_2(m, U_2, U_3), &L<x<\ell_n,\\
-d_3U_3''=f_3(U_2, U_3), &L<x<\ell_n,\\
U_2=U_3=0, & x=L, \ell_n
  \end{cases}\eess
has a unique positive solution, denoted by $(U_{2n}, U_{3n})$. The comparison principle gives
 \bes
 (U_{2l_n}, U_{3l_n})\ge (U_{2n}, U_{3n})\;\;\;{\rm in}\;\;[L, \ell_n].
 \lbl{2.41}\ees
Similar to Lemma \ref{l2.2} we can show that $\lim_{n\to\yy}(U_{2n}, U_{3n})=(U_2^L, U^L_3)$ in $[C^2_{\rm loc}([L,\yy))]^2$,
and $(U^L_2, U^L_3)$ is the unique positive solution of
 \bess\begin{cases}
-d_2U_2''=f_2(m, U_2, U_3), &L<x<\yy,\\
-d_3U_3''=f_3(U_2, U_3), &L<x<\yy,\\
U_2(L)=U_3(L)=0.
  \end{cases}\eess
This combines with \qq{2.41} indicates that $U_2, U_3>0$ in $[L,\yy)$. We have a contradiction.
\end{proof}

\section{Long time behaviors of solution component $(u_1,u_2,u_3)$}

Throughout this section, the positive constant $C_0$ was given in Theorem \ref{th1.1}.

\begin{theo}\lbl{th3.1a} Let $(u_1,u_2,u_3,h)$ be the unique global solution of \qq{1.3}. If $\mathcal{R}_0={\theta bk}/(acq)\le 1$, then
 \bess
 &\dd\lim_{t\to\yy}u_1(t,x)=\ol U_1(x) \;\;{\rm locally~uniformly~in~} [0,\yy),&\\
 &\dd\lim_{t\to\yy}u_2(t,x)=0,\;\; \lim_{t\to\yy}u_3(t,x)=0 \;\;{\rm uniformly~in~} [0,\yy),&
 \eess
where $\ol U_1(x)$ is the unique bounded positive solution of \qq{2.2}.
  \end{theo}

\begin{proof} Firstly, let $\ol u_1(t)$ be the unique solution of
 \bess
 \ol u'_1(t)=\theta -a\ol u_1,\;\;t>0;\;\;\;
  \ol u_1(0)=C_0.
 \eess
Then $\lim_{t\to\yy}\ol u(t)=\theta/a$, and $u(t,x)\le\ol u(t)$ in $[0,\yy)\times[0, h(t)]$ by the comparison principle. So,
$\limsup_{t\to\yy}u_1(x,t)\le\theta/a$ uniformly in $\mathbb{R}$. For any given $0<\ep\ll 1$, there exists $T_\ep\gg1$ such that $u(t,x)<\theta/a+\ep$ for all $t\ge T_\ep$ and $x\in\mathbb{R}$. Thus we have
  \bess\begin{cases}
\partial_t u_2-d_2\partial_{xx}u_2\le f_2(\theta/a+\ep, u_2,u_3) &t>T_\ep, \; 0<x<h(t),\\
\partial_t u_3-d_3\partial_{xx}u_3=f_3(u_2,u_3), &t>T_\ep, \; 0<x<h(t),\\
u_2=u_3=0, &t>0, \ x\notin(0,h(t)),\\
u_2(T_\ep,x)\le C_0,\;\; u_3(T_\ep,x)\le C_0, &x\in(0,\yy).
 	\end{cases}\eess
Let $(z^\ep_2(t), z^\ep_3(t))$ be the unique positive solution of
 \bess\begin{cases}
(z^\ep_2)'=f_2(\theta/a+\ep, z^\ep_2, z^\ep_3), &t>T_\ep,\\
(z^\ep_3)'=f_3(z^\ep_2, z^\ep_3), &t>T_\ep,\\
z^\ep_2(T_\ep)=z^\ep_3(T_\ep)=C_0.
  \end{cases}\eess
Then either $\lim_{t\to\yy}(z^\ep_2(t), z^\ep_3(t))=(v_\ep, w_\ep)$, where $(v_\ep, w_\ep)$ is the unique positive solution of
 \bes
 f_2(\theta/a+\ep, v_\ep, w_\ep)=0,\;\;\;
 f_3(v_\ep, w_\ep)=0,
 \lbl{3.1}\ees
or $\lim_{t\to\yy}(z^\ep_2(t), z^\ep_3(t))=(0, 0)$ if \qq{3.1} does not have positive solution. When $\mathcal{R}_0<1$, we can choose $\ep>0$ is small such that \qq{3.1} has no positive solution. When $\mathcal{R}_0=1$, then \qq{3.1} has a unique positive solution $(v_\ep, w_\ep)$ and $\lim_{\ep\to 0^+}( v_\ep, w_\ep)=(0,0)$.

By the comparison argument, $(u_2(t,x), u_3(t,x))\le(z^\ep_2(t), z^\ep_3(t))$ for all $t\ge T_\ep$ and $x>0$. It follows that $\limsup_{t\to\yy}(u_2(t,x), u_3(t,x))\le(0, 0)$ uniformly in $[0,\yy)$. Thus, by the first differential equations of \qq{1.3}, $\lim_{t\to\yy}u_1(t,x)=\ol U_1(x)$  locally uniformly in $[0,\yy)$. \end{proof}

\begin{theo}\lbl{thy.1} Let $(u_1,u_2,u_3,h)$ be the unique global solution of \qq{1.3}. If $\mathcal{R}_0={\theta bk}/(acq)>1$ and $\mathcal{R}_0+\sqrt{\mathcal{R}_0}>b/a$, then
\bes
 \liminf_{t\to\infty}u_i(t,x)\geq\ud U_i(x),\;\;\limsup_{t\to\infty}u_i(t,x)\leq\ol U_i(x)\;\; \mbox{locally uniformly in}\;\, [0, \infty),\;\;i=1,2,3,
 \qquad\lbl{3.2} \ees
where $\ol U_1$, $(\ol U_2, \ol U_3)$, $\ud U_1$ and $(\ud U_2, \ud U_3)$ are the unique bounded positive solutions of \qq{2.2}, \qq{2.24}, \qq{2.25} and \qq{2.27}, respectively.
  \end{theo}

\begin{proof} In this proof, $C_0$ is the positive constant given by Theorem \ref{th1.1}.

{\it Step 1}. Take
 \[\phi(x)=\left\{\begin{array}{ll}
  u_0(x), \ \ & 0\leq x\leq h_0,\\
  0, \ \ & x\geq h_0,
  \end{array}\right.\]
Let $v$ be the unique solution of
 \bess\bbb
 v_t-d_1v_{xx}=\theta-a v,&t>0,\;0<x<\yy,\\
 v(t,0)=0, &t>0,\\
 v(0,x)=\phi(x), &0\leq x<\infty
 \nnn\eess
Then $u_1\le v$ for $t>0$ and $0\leq x\leq h(t)$ by the comparison principle. As $\lim_{t\to\infty}v(t,x)=\ol U_1(x)$ locally uniformly in $[0, \infty)$. Hence, the second inequality of \qq{3.2} holds for $i=1$.

{\it Step 2}. For the given $l>0$ and $0<\ep\ll 1$, by use of the upper and lower solutions method it can be shown that the boundary value problem
 \bes\begin{cases}
-d_2 W_{l,\ep}''=f_2(\ol U_1(x)+\ep, W_{l,\ep}, Z_{l,\ep}),\;\;\;\,0<x<l,\\
-d_3Z_{l,\ep}''=f_3(W_{l,\ep}, Z_{l,\ep}),\hspace{23mm}0<x<l,\\
W_{l,\ep}(0)=Z_{l,\ep}(0)=0,\;\;W_{l,\ep}(l)=Z_{l,\ep}(l)=C_0+1
 	\end{cases}\lbl{3.3}\ees
has at least one positive solution, and similar to the proofs of Lemmas \ref{l2.2} and \ref{l2.3} it can be proved that the positive solution of \qq{3.3} is unique, denoted by $(W_{l,\ep}, Z_{l,\ep})$. It is clear that $\lim_{\ep\to 0}(W_{l,\ep}, Z_{l,\ep})=(W_l, Z_l)$ in $[C^2([0,l])]^2$, where $(W_l, Z_l)$ is the unique positive solution of
 \bess\begin{cases}
-d_2W_l''=f_2(\ol U_1(x), W_l, Z_l),\;\;\;\,0<x<l,\\
-d_3Z_l''=f_3(W_l, Z_l),\hspace{17mm}0<x<l,\\
W_l(0)=Z_l(0)=0,\;\;W_l(l)=Z_l(l)=C_0+1.
 	\end{cases}\eess
It is easy to show that $W_l, Z_l\le C_0+1$ in $[0,l]$. Similar to Lemma \ref{l2.2}, $\lim_{l\to\yy}(W_l, Z_l)=(\ol U_2, \ol U_3)$ in $[C^2_{\rm loc}([0,\yy))]^2$, and $(\ol U_2, \ol U_3)$ is the unique bounded positive solution of \qq{2.24}. For any given $L>0$ and $\sigma>0$, there exist $l>L$ and $0<\ep\ll 1$ such that
 \bes
 (W_{l,\ep}, Z_{l,\ep})\le (\ol U_2+\sigma, \ol U_3+\sigma)
 \;\;\;{\rm in}\;\;[0, L].
 \lbl{3.4}\ees

For the $\ep$ and $l$ identified above. Since $\limsup_{t\to\infty}u_1(t,x)\leq\ol U_1(x)$ locally uniformly in $[0, \infty)$, there exists a large $T_\ep^l$, such that
 \[h(t)>l, \;\;\; u_1(t,x)<\ol U_1(x)+\ep\ff 0\leq x\leq l, \ t\geq T_\ep^l.\]
Noticing that $u_i(T_\ep^l,0)=0$, $u_i\le C_0$, $u_i(T_\ep^l,\cdot)\in C^1([0,l])$, $i=2,3$. We can find two functions $\eta_i(x)\in C^{2+\alpha}([0,l])$, such that
 \[\eta_i(0)=0,\;\;\eta_i(l)=C_0+1,\;\;\;{\rm and}\;\; \eta_i(x)\ge u_i(T_\ep^l,x)\;\;\;{\rm in}\;\;[0,l],\;\;\;i=2,3.\]
Let $(w_{l,\ep}, z_{l,\ep})$ be the unique positive solution of the following initial-boundary value problem
 \bess\begin{cases}
\partial_t w_{l,\ep}-d_2\partial_{xx}w_{l,\ep}=f_2(\ol U_1(x)+\ep, w_{l,\ep}, z_{l,\ep}), &t>T_\ep, \ 0<x<l,\\
\partial_t z_{l,\ep}-d_3\partial_{xx}z_{l,\ep}=f_3(w_{l,\ep},z_{l,\ep}), &t>T_\ep^l, \ 0<x<l,\\
w_{l,\ep}(t,0)=z_{l,\ep}(t,0)=0,\;\;w_{l,\ep}(t,l)=z_{l,\ep}(t,l)=C_0+1, &t\ge T_\ep^l, \\
w_{l,\ep}(T_\ep^l,x)=\eta_2(x),\;\; z_{l,\ep}(T_\ep^l,x)=\eta_3(x), &x\in(0,l).
 	\end{cases}\eess
Then $(u_2, u_3)\le(w_{l,\ep}, z_{l,\ep})$ by the comparison principle, and $\lim_{t\to\yy}(w_{l,\ep}, z_{l,\ep})=(W_{l,\ep}, Z_{l,\ep})$ in $C^2([0,l])$ as  $(W_{l,\ep}, Z_{l,\ep})$ is the unique positive solution of \qq{3.3}. This combined with \qq{3.4} yields
 \bess
  \limsup_{t\to\yy}(u_2, u_3)\le\lim_{t\to\yy}(w_{l,\ep}, z_{l,\ep})
  \le (\ol U_2+\sigma, \ol U_3+\sigma) \;\;\;{\rm in}\;\;[0, L].
  \eess
The arbitrariness of $\sigma$ and $L$ indicates that the second inequality of \qq{3.2} holds for $i=2,3$.

{\it Step 3}. For the given $l>0$ and $0<\ep\ll 1$, it is easy to show that the problem
 \bess\begin{cases}
-d_1V_{l,\ep}''=f_1(V_{l,\ep},\ol U_3+\ep), \; &0<x<l,\\
V_{l,\ep}=0, & x=0,\,l
  \end{cases}\eess
has a unique positive solution $V_{l,\ep}$, and $\lim_{\ep\to 0}V_{l,\ep}=V_l$ in $C^2([0,l])$, where $V_l$ is the unique positive solution of
 \bess\begin{cases}
-d_1V_l''=f_1(V_l,\ol U_3), \; &0<x<l,\\
V_l=0, & x=0,\,l.
  \end{cases}\eess
Moreover, $\lim_{l\to\yy}V_l=\ud U_1$ in $C^2_{\rm loc}([0,\yy))$, where $\ud U_1$ is the unique positive solution of \qq{2.25}. Similar to the above, we can show that the first inequality of \qq{3.2} holds for $i=1$.

{\it Step 4}. Noticing that the condition $\mathcal{R}_0+\sqrt{\mathcal{R}_0}>b/a$ is equivalent to $bk\frac{\theta\sqrt{\mathcal{R}_0}}{(a+b)\sqrt{\mathcal{R}_0}-b}>cq$. There exists $0<\ep_0<\frac{\theta\sqrt{\mathcal{R}_0}}{(a+b)\sqrt{\mathcal{R}_0}-b}$ such that
 \bes
 bk\kk(\frac{\theta\sqrt{\mathcal{R}_0}}{(a+b)\sqrt{\mathcal{R}_0}-b}-\ep\rr)>cq,\;\;\forall\; 0<\ep<\ep_0.\lbl{3.5a}\ees
For the given $0<\ep<\ep_0$, we define
 \[\rho_\ep(x)=\left\{\begin{array}{ll}
 0,\;&\ud U_1(x)<\ep,\\
 \ud U_1(x)-\ep, \;\;&\ud U_1(x)\ge\ep.
 \end{array}\rr.\]
Then $\rho_\ep\in C^\alpha_{\rm loc}([0,\yy))$ and $\lim_{x\to\yy}\rho_\ep(x)=\frac{\theta\sqrt{\mathcal{R}_0}}{(a+b)\sqrt{\mathcal{R}_0}-b}-\ep>0$ since $\lim_{x\to\infty}\ud U_1(x)=\frac{\theta\sqrt{\mathcal{R}_0}}{(a+b)\sqrt{\mathcal{R}_0}-b}$.
Notice that \qq{3.5a} holds, we can use Lemma \ref{l2.1} to conclude that there exists $l^*>0$ such that, when $l>2l^*$, the problem
 \bes\begin{cases}
-d_2\Phi_{l,\ep}''=f_2(\rho_\ep(x),\Phi_{l,\ep},\Psi_{l,\ep}), \;& 0<x<l,\\
-d_3\Psi_{l,\ep}''=f_3(\Phi_{l,\ep},\Psi_{l,\ep}),&0<x<l,\\
\Phi_{l,\ep}=\Psi_{l,\ep}=0,\;\;&x=0,\;l
  \end{cases}\label{3.5}\ees
has a unique positive solution $(\Phi_{l,\ep}, \Psi_{l,\ep})$. Since $\rho_\ep'\in L^\yy((0,l))$ and $\|\rho_\ep'\|_{L^\yy((0,l))}\le\|\ud U_1'\|_{L^\yy((0,l))}$, it is easy to see that $\|(\Phi_{l,\ep}, \Psi_{l,\ep})\|_{C^{2+\alpha}([0,l])}$ is bounded respect to $0<\ep<\ep_0$. Then  $\lim_{\ep\to 0}(\Phi_{l,\ep}, \Psi_{l,\ep})=(\Phi_l, \Psi_l)$ in $[C^2([0,l])]^2$, where $(\Phi_l, \Psi_l)$ is the unique positive solution of
 \bess\begin{cases}
-d_2\Phi_l''=f_2(\ud U_1(x), \Phi_l, \Psi_l),\; & 0<x<l,\\
-d_3\Psi_l''=f_3(\Phi_, \Psi_l),\; &0<x<l,\\
\Phi_l=\Psi_l=0,\;\;&x=0,\;l.
 	\end{cases}\eess
Similar to the proof of Lemma \ref{l2.2}, it can be deduced that
$\lim_{l\to\yy}(\Phi_l, \Psi_l)=(\ud U_2, \ud U_3)$ in $[C^2_{\rm loc}([0,\yy))]^2$, and $(\ud U_2, \ud U_3)$ is the unique bounded positive solution of \qq{2.27}. For any given $L>0$ and $\sigma>0$, there exist $l>L$ and $0<\ep\ll 1$ such that
 \bes
 (\Phi_{l,\ep}, \Psi_{l,\ep})\ge (\ud U_2-\sigma, \ud U_3-\sigma)\;\;\;{\rm in}\;\;[0, L].
 \lbl{3.6}\ees

For the $\ep$ and $l$ identified above. Since $\lim_{t\to\yy}h(t)=\yy$ and $\liminf_{t\to\infty}u_1(t,x)\geq\ud U_1(x)$ locally uniformly in $[0, \infty)$,  there exists a large $T_\ep^l$, such that
 \[h(t)>l, \ \ \ u_1(t,x)>\ud U_1(x)-\ep\ff 0\leq x\leq l, \; t\geq T_\ep^l.\]
Certainly, $u_1(t,x)\ge\rho_\ep(x)$ for $0\leq x\leq l$ and $t\geq T_\ep^l$.  Noticing that $u_i(T_\ep^l,0)=0$, $u_i(T_\ep^l,x)>0$ in $(0,l]$, $\partial_xu_i(T_\ep^l,0)>0$ and $u_i(T_\ep^l,\cdot)\in C^1([0,l])$, $i=2,3$. We can find two functions $\gamma_i(x)\in C^{2+\alpha}([0,l])$, such that
 \[\gamma_i(0)=\gamma_i(l)=0,\;\;\;{\rm and}\;\; 0<\gamma_i(x)\le u_i(T_\ep^l,x)\;\;\;{\rm in}\;\;(0,l),\;\;\;i=2,3.\]
Let $(\varphi_{l,\ep}, \psi_{l,\ep})$ be the unique positive solution of the following initial-boundary value problem
 \bess\begin{cases}
\partial_t\varphi_{l,\ep}-d_2\partial_{xx}\varphi_{l,\ep}=f_2(\rho_\ep(x), \varphi_{l,\ep}, \psi_{l,\ep}), &t>T_\ep, \ 0<x<l,\\
\partial_t\psi_{l,\ep}-d_3\partial_{xx}\psi_{l,\ep}=f_3(\varphi_{l,\ep}, \psi_{l,\ep}), &t>T_\ep^l, \ 0<x<l,\\
\varphi_{l,\ep}=\psi_{l,\ep}=0, &t\ge T_\ep^l,\; x=0,\,l, \\
\varphi_{l,\ep}(T_\ep^l,x)=\gamma_2(x),\;\; \psi_{l,\ep}(T_\ep^l,x)=\gamma_3(x), &x\in(0,l).
 	\end{cases}\eess
Then $(u_2, u_3)\ge(\varphi_{l,\ep}, \psi_{l,\ep})$ by the comparison principle, and $\lim_{t\to\yy}(\varphi_{l,\ep}, \psi_{l,\ep})=(\Phi_{l,\ep}, \Psi_{l,\ep})$ in $C^2([0,l])$ since $(\Phi_{l,\ep}, \Psi_{l,\ep})$ is the unique positive solution of \qq{3.5}. This combined with \qq{3.6} yields
 \bess
  \limsup_{t\to\yy}(u_2, u_3)\ge\lim_{t\to\yy}(\varphi_{l,\ep}, \psi_{l,\ep})
  \ge (\ud U_2-\sigma, \ud U_3-\sigma) \;\;\;{\rm in}\;\;[0, L].
  \eess
The arbitrariness of $\sigma$ and $L$ indicates that the first inequality of \qq{3.2} holds for $i=2,3$.
\end{proof}

\end{document}